\documentclass[12pt]{amsart}
\usepackage{amssymb,verbatim,enumerate,ifthen}
\usepackage{esint}
\usepackage[mathscr]{eucal}
\usepackage[utf8]{inputenc}
\usepackage[T1]{fontenc}
\makeatletter
\@namedef{subjclassname@2010}{%
  \textup{2010} Mathematics Subject Classification}
\makeatother

\newtheorem{thm}{Theorem}
\newtheorem{cor}{Corollary}
\newtheorem{prop}{Proposition}
\newtheorem{lem}{Lemma}

\newtheorem*{xrem}{Remark}

\numberwithin{equation}{section}

\newcommand\nolabel[1]{\nonumber}

\newcommand\R{\mathbb{R}}

\newcommand\N{\mathbb{N}}
\newcommand\Q{\mathbb{Q}}
\newcommand\Z{\mathbb{Z}}

\newcommand{\abs}[1]{\left| #1 \right| }

\DeclareMathOperator{\conv}{conv}
\DeclareMathOperator{\dist}{dist}

\numberwithin{equation}{section}
\def\eq#1{{\rm(\ref{#1})}}
\def\Eq#1#2{\ifthenelse{\equal{#1}{*}}
  {\begin{equation*}\begin{aligned}[]#2\end{aligned}\end{equation*}}
  {\begin{equation}\begin{aligned}[]\label{#1}#2\end{aligned}\end{equation}}}

\title{Jensen-type geometric shapes}

\author[P. Pasteczka]{Pawe\l{} Pasteczka}
\address{Institute of Mathematics \\ Pedagogical University of Cracow \\ Podchor\k{a}\.zych str. 2, 30-084 Krak\'ow, Poland}
\email{pawel.pasteczka@up.krakow.pl}

\begin{document}
\maketitle
\begin{abstract}
We present both necessary and sufficient conditions to the convex closed shape $X$ such that the inequality 
$$ \frac{1}{|X|} \int_X f(x)\:dx \le \frac{1}{|\partial X|} \int_{\partial X} f(x)\:dx$$
is valid for every  convex function $f \colon X \to \mathbb{R}$ ($\partial X$ stands for the boundary of $X$).

It is proved that this inequality holds if $X$ is 
(i) an $n$-dimensional parallelotope, (ii) an $n$-dimensional ball, (iii) a convex polytope having an inscribed sphere (tangent to all its facets) with center in the center of mass of $\partial X$.
\end{abstract}

\section{Introduction}

Dragomir and Pearce proved \cite[Theorem 215]{DraPea00} that if $B_n$ is an $n$-dimensional ball then
\Eq{*}{
\fint_{B_3} f(x) dx \le \fint_{\partial B_3} f(x) dx
}
for every convex function $f$; here and below $\fint$ stands for the average integral (more precisely $\fint_X f(x)\:dx:=\tfrac1{|X|}\int_X f(x)\:dx$). 
During \emph{Conference on Inequalities and Applications 2016} P\'ales stated the problem whether for every convex and closed set $X$ and every convex function $f \colon X \to \R$, the inequality
\Eq{E:conv_def}{
\fint_X f(x) dx \le \fint_{\partial X} f(x) dx
}
is valid. It is however easy to verify that for the triangle $T$ with vertices $(0,-1)$, $(0,1)$, $(1,0)$ and the function $f \colon T \ni (x,y) \mapsto x$, the inequality \eq{E:conv_def} voids (as the inequality $\tfrac13 \le 1-\tfrac{\sqrt{2}}2$ is not valid). Furthermore, this property is invariant under transition, scaling, changing orientation and reflection, whence it is a property of a shape. 

This motivates us to introduce the following definition. Convex and closed shape $X$ is called \emph{Jensen-type} if for every convex function $f \colon X \to \R$, the inequality \eq{E:conv_def} is satisfied.

Using this definition Dragomir--Pearce result can be expressed briefly as {\it $3$-dimensional ball is of Jensen-type} or {\it $B_3$ is of Jensen-type}. The second example can be expressed by {\it 45--45--90 triangle is not of Jensen-type}.

Motivated by these preliminaries we are going to prove this property for \mbox{regular polygons}, \mbox{parallelotopes} (in all dimensions), \mbox{balls} (in all dimensions), and \mbox{Platonic solids}.

\section{Results}
We begin with some necessary condition for $X$ to be Jensen-type. It repeats the argumentation which was already presented in a case of triangle $T$.

\begin{lem}\label{lem:necc}
 If $X$ is of Jensen-type then centers of mass of $X$ and $\partial X$ coincide.
\end{lem}

\begin{proof}
Let $\pi_i \colon \R^n \to \R$ be a projection on $i$-th coordinate ($i \in \{1,\dots,n\}$). Both $\pi_i$ and $-\pi_i$ are convex so, as $X$ is of Jensen-type, we get
 \Eq{*}{
 \fint_X \pi_i(x) dx \le \fint_{\partial X} \pi_i(x) dx
\quad\text{ and }\quad \fint_X -\pi_i(x) dx \le \fint_{\partial X} -\pi_i(x) dx.
 }
Thus $\fint_X \pi_i(x) dx = \fint_{\partial X} \pi_i(x) dx$ for $i \in \{1,\dots,n\}$. But centers of mass of $X$ and $\partial X$ equal to $\left(\fint_X \pi_i(x) dx \right)_{i=1}^n$ and $\left(\fint_{\partial X} \pi_i(x) dx \right)_{i=1}^n$, respectively. The above equality states that these points coincide.
\end{proof}

\begin{xrem}
We have presented some necessary condition for a shape to be of Jensen-type. Our conjecture is that every convex shape which satisfies this condition is of Jensen-type.
\end{xrem}

In the subsequent result we are going to prove that all parallelotopes and $n$-dimensional balls are of Jensen-type.

\begin{prop}
 All parallelotopes are of Jensen-type.
\end{prop}

\begin{proof}
Fix parallelotope $W$ of dimension $n$. Let $\{S_i\}_{i=1}^{2^n}$ be its all facets. Denote by $S_i^*$ the facet opposite to $S_i$. In fact facet $S_i^*$ is a facet $S_i$ shifted by some vector $v_i \in \R^n$.
Finally, for $y \in S_i$, let $y^\ast:=y+v_i \in S_i^*$.

Now fix a convex function $f \colon W \to \R$. By Hermite-Hadamard inequality we have
\Eq{*}{
%\fint_y^{y^*} f(x) dx &\le \frac{f(y)+f(y^*)}2,\\
\int_y^{y^*} f(x) dx &\le \dist(S_i,S_i^*) \cdot \frac{f(y)+f(y^*)}2 \qquad (y \in S_i),
}
which in view of the equality $|S_i|\dist(S_i,S_i^*)=|W|$ is equivalent to
\Eq{*}{
\frac{2|S_i|}{|W|} \int_y^{y^*} f(x) dx &\le f(y)+f(y^*) \qquad (y \in S_i),
}
We integrate both side over $S_i$ to obtain
\Eq{*}{
2|S_i| \cdot \fint_W f(x)\: dx \le \int_{S_i \cup S_i^*} f(x) \:dx.
}
Finally, let us sum up the above inequality for $i \in \{1,2,\dots,2^n\}$. Then we obtain
\Eq{*}{
2\abs{\partial W} \cdot \fint_W f(x) dx &\le 2 \int_{\partial W} f(x) dx,
}
which simplifies to $\fint_W f(x) dx \le \fint_{\partial W} f(x) dx$.
\end{proof}

In the next proposition we will generalize the Dragomir-Pearce result.
\begin{prop}
The $n$-dimensional ball is of Jensen-type for every $n \ge 2$. 
\end{prop}
\begin{proof}
Fix a convex function $f \colon B_n \to \R$. 
We have 
\Eq{*}{
\fint_{B_n} f(x) dx &= \frac{1}{|B_n|} \int_{B_n} f(x) dx 
= \frac{1}{|B_n|} \int_0^1 r^{n-1} \int_{S_{n-1}} f(rx) dx dr \\
&= \frac{1}{|B_n|} \int_0^1 \frac{r^{n-1}}{2} \int_{S_{n-1}} f(rx)+f(-rx) dx dr
}
Applying Wright-convexity of $f$ we get
\Eq{*}{
\fint_{B_n} f(x) dx &\le \frac{1}{|B_n|} \int_0^1 \frac{r^{n-1}}{2} \int_{S_{n-1}} f(x)+f(-x) dx dr \\
&\le \frac{1}{|B_n|} \int_0^1 r^{n-1} dr \int_{S_{n-1}} f(x) dx 
= \frac{1}{n |B_n|} \int_{S_{n-1}} f(x) dx.
}
By the identity $n |B_n|=\abs{S_{n-1}}$, we obtain desired inequality.
\end{proof}

\subsection{Convex polytopes having an inscribed sphere} 
We will now struggle with convex polytopes. To avoid misunderstandings the \emph{inscribed sphere} is the sphere which is tangent to all facets.

\begin{lem}
\label{lem:ostroslup}
Let $n \in \N$, $\Delta \subset \R^n$ be a convex $(n-1)$-dimensional shape, $s \in \R^n \setminus \Delta$ and $G=\conv \{\Delta,s\}$. Then for every convex function $f \colon G \to \R$,
\Eq{*}{
\fint_G f(x) dx \le \frac{n}{n+1} \fint_{\Delta} f(x) dx + \frac1{n+1} f(s).
}

\end{lem}

\begin{proof}
For each $\theta \in (0,1]$ let $T_\theta$ a homothetic transformation of $\Delta$ with center $s$ and scale $\theta$. Denote its image by $\Delta_\theta$. Moreover denote $H:=\dist(p,\Delta)$ and $\pi \colon \Delta \to \Delta_1$ be a projection such that $\pi|_{\Delta_\theta}=T_\theta^{-1}$.

We know that 
\Eq{*}{
x=\theta \cdot \pi(x)+(1-\theta) \cdot s \text{ for all }\theta \in (0,1] \text{ and }x \in \Delta_\theta.
}
Whence
\Eq{*}{
\int_G f(x) dx 
&= H \cdot \int_0^1 \int_{\Delta_\theta} f(x) dx d\theta \\
&= H \cdot \int_0^1 \int_{\Delta_\theta} f(\theta \cdot \pi(x) +(1-\theta) s) dx d\theta. \\
&= H \cdot \int_0^1 \int_{\Delta_1} \theta^{n-1} f(\theta \cdot x +(1-\theta) s) dx d\theta 
}
Thus, by Jensen's and Fubini's inequalities,
\Eq{*}{
\int_G f(x) dx  &\le H \cdot  \int_{\Delta_1} \int_0^1 \theta^n \: d\theta \cdot f(x) +\int_0^1 \theta^{n-1} (1-\theta) \: d\theta \cdot f(s) \: dx.
}
By $\int_0^1 \theta^n d\theta =\tfrac{1}{n+1}$ and $\int_0^1 \theta^{n-1}(1-\theta)d\theta=\tfrac{1}{n(n+1)}$ we obtain
\Eq{*}{
\int_G f(x) dx  &\le H \cdot \left( \tfrac{1}{n+1} \int_{\Delta_1} f(x) dx + \tfrac{1}{n(n+1)}  \abs{\Delta_1} \cdot f(s) \right) \\
&= \frac{H \cdot \abs{\Delta_1}}n \cdot \left( \tfrac{n}{n+1} \fint_{\Delta_1} f(x) dx + \tfrac{1}{n+1} \cdot f(s) \right).
}
To finish the proof we can use the classical equality $|G|=\tfrac1n \cdot H \cdot |\Delta_1|$.
\end{proof}

\begin{thm}
 Let $W$ be an $n$-dimensional convex polytope having an inscribed sphere with center $s$. Then 
\Eq{E:thm1}{
\fint_W f(x) dx \le \frac{n}{n+1} \fint_{\partial W} f(x) dx + \frac1{n+1} f(s)
}
for every convex function $f \colon W \to \R$.
\end{thm}

\begin{proof}
Let $r$ be a radius of the inscribed sphere. Denote all facets of $W$ by $\{A_1,\dots,A_k\}$, moreover let $G_i=\conv\{A_i,s\}$. We have
$\abs{G_i}=\tfrac rn \cdot \abs{A_i}$, in particular $\abs{W}=\tfrac rn \abs{\partial W}$.
By Lemma~\ref{lem:ostroslup}, for all $i \in \{1,\dots,k\}$ we have
\Eq{*}{
%\fint_{G_i} f(x) dx &\le \frac{n}{n+1} \fint_{A_i} f(x) dx + \frac1{n+1} f(s) \\
\int_{G_i} f(x) dx &\le \frac{n}{n+1} \cdot \abs{G_i} \cdot \fint_{A_i} f(x) dx + \frac1{n+1} \cdot \abs{G_i} \cdot f(s) \\
&= \frac{n}{n+1} \cdot \frac rn \cdot \abs{A_i} \cdot \fint_{A_i} f(x) dx + \frac1{n+1} \cdot \abs{G_i} \cdot f(s) \\
&= \frac{r}{n+1} \cdot \int_{A_i} f(x) dx + \frac1{n+1} \cdot \abs{G_i} \cdot f(s)
}
Summing this inequality (side-by-side for $i \in \{1,\cdots,n\}$) we obtain
\Eq{*}{
\int_{W} f(x) dx &\le \frac{r}{n+1} \cdot \int_{\partial W} f(x) dx + \frac1{n+1} \cdot \abs{W} \cdot f(s) \\
&= \frac{r\cdot \abs{\partial W}}{n+1} \cdot \fint_{\partial W} f(x) dx + \frac1{n+1} \cdot \abs{W} \cdot f(s)
}
To finish the proof note that $\frac{r\cdot \abs{\partial W}}{n+1} = \frac{n}{n+1} \cdot \frac{r\cdot \abs{\partial W}}n = \frac{n}{n+1} \cdot \abs{W}$.
\end{proof}

We can now present some simple corollary.
\begin{cor}
Let $W$ be a convex $n$-dimensional polytope having an inscribed sphere with center $s$ and $m$ be the center of mass of $\partial W$. Then 
\Eq{*}{
\fint_W f(x) dx \le \fint_{\partial W} f(x) dx + \frac1{n+1} (f(s)-f(m))
}
for every convex function $f \colon W \to \R$.
\end{cor}

Indeed, by the Jensen's inequality we have $f(m) \le \fint_{\partial W} f(x)\:dx$, thus $0\le \tfrac{1}{n+1} (\fint_{\partial W} f(x)\:dx - f(m) )$. We can now sum this inequality with \eq{E:thm1} side-by-side to obtain desired inequality.

As a trivial particular case we obtain some sufficient condition for $W$ to be of Jensen-type.
\begin{thm}
Let $W$ be a convex polytope having an inscribed sphere. If the center of this sphere coincide with the center of mass $\partial W$, then $W$ is of Jensen-type.
\end{thm}

Obviously this result implies that all Platonic solids are of Jensen-type. 

%;
%%https://www.geometrycode.com/free/polyhedra-math-tables-for-platonic-and-archimedean-solids/ -- platonic and archimedian
%http://www.deltami.edu.pl/temat/matematyka/geometria/stereometria/2013/12/30/Lampy_Catalana/
%W każdą z nich można wpisać kulę styczną do wszystkich ścian. 
%+ potrzebna nietrywialna grupa obrotów.
%in dimension three: .
%\end{prop}
\subsection*{Acknowledgement}
I am grateful to Karol Gryszka and Alfred Witkowski for their valuable remarks.

\def\cprime{$'$} \def\R{\mathbb R} \def\Z{\mathbb Z} \def\Q{\mathbb Q}
  \def\C{\mathbb C}


\begin{thebibliography}{1}

\bibitem{DraPea00}
S.~S. Dragomir and C.~E.~M. Pearce.
\newblock {\em {{S}elected {T}opics on {H}ermite-{H}adamard {I}nequalities}}.
\newblock RGMIA Monographs ({\tt
  http://rgmia.vu.edu.au/monographs/hermite\_hadamard.html}), Victoria
  University, 2000.

\end{thebibliography}
\end{document}